\documentclass[12pt,leqno,a4paper]{amsart}
\usepackage{amsmath, amsthm, verbatim, amsfonts, amssymb, color}
\usepackage{mathrsfs}
\usepackage{dsfont}
\usepackage{verbatim}
\usepackage[a4paper,margin=1in]{geometry}

\theoremstyle{plain}
\newtheorem{theorem}{Theorem}[section]
\newtheorem{corollary}[theorem]{Corollary}
\newtheorem{lemma}[theorem]{Lemma}
\newtheorem{proposition}[theorem]{Proposition}

\theoremstyle{definition}
\newtheorem{remark}[theorem]{Remark}
\newtheorem{example}[theorem]{Example}

\numberwithin{equation}{section}

\newcommand\E{\mathds E}
\newcommand\R{\mathds R}
\newcommand\W{\mathds W}
\newcommand\N{\mathds N}
\newcommand\Pp{\mathds P}
\newcommand\Ss{\mathds S}

\newcommand\Bscr{\mathscr{B}}
\newcommand\Dscr{\mathscr{D}}
\newcommand\Cscr{\mathscr{C}}
\newcommand\Escr{\mathscr{E}}
\newcommand\Fscr{\mathscr{F}}
\newcommand\I{\mathds 1}

\renewcommand\d{\mathrm{d}}
\newcommand\e{\mathrm{e}}

\newcommand{\cmspace}{\mathds{H}_M}

\begin{document}\allowdisplaybreaks
\title[A Cameron--Martin Type Quasi-Invariance Theorem]{\bfseries On a Cameron--Martin Type Quasi-Invariance
Theorem and Applications to Subordinate Brownian Motion}

\author[C.-S.~Deng]{Chang-Song Deng}
\address[C.-S.~Deng]{School of Mathematics and Statistics\\ Wuhan University\\ Wuhan 430072, China}
\curraddr{TU Dresden\\ Fachrichtung Mathematik\\ Institut f\"{u}r Mathematische Stochastik\\ 01062 Dresden, Germany}
\email{dengcs@whu.edu.cn}
\thanks{The first-named author gratefully acknowledges support through the Alexander-von-Humboldt foundation, the National Natural Science Foundation of China (11401442) and the International Postdoctoral Exchange Fellowship Program (2013)}

\author[R.\,L.~Schilling]{Ren\'e L.\ Schilling}
\address[R.\,L.~Schilling]{TU Dresden\\ Fachrichtung Mathematik\\ Institut f\"{u}r Mathematische Stochastik\\ 01062 Dresden, Germany}
\email{rene.schilling@tu-dresden.de}

\subjclass[2010]{60J75, 60G17}
\keywords{Cameron--Martin theorem, quasi-invariance, integration by parts formula, subordinate Brownian motion}

\date{\today}

\maketitle

\begin{abstract}
    We present a Cameron--Martin type quasi-invariance theorem for subordinate Brownian motion. As applications, we establish an integration by parts formula and construct a gradient operator on the path space of subordinate Brownian motion, and we obtain some canonical Dirichlet forms. These findings extend the corresponding classical results for Brownian motion.
\end{abstract}

\section{Introduction}\label{sec1}

    Recently the stability of properties of Markov processes and their semigroups under subordination in the sense of Bochner has attracted great interest. In \cite{GRW}, Wang's dimension free Harnack inequality was established for a class of subordinate semigroups. Nash and Poincar\'{e} inequalities are preserved under subordination, cf. \cite{SW,GM}. In our recent paper \cite{DS14}, we show that shift Harnack inequalities (in the sense of \cite{Wan14}) remain valid under subordination in the sense of Bochner. It is a natural question whether further probabilistic properties, e.g.\ quasi-invariance, are preserved by subordination.

    The Cameron--Martin theorem, which was discovered by R.H.~Cameron and W.T.~Martin \cite{CM} (see e.g.~\cite{Dri, Hs95, Hsu} and the references therein for further developments), plays a fundamental role in the analysis on the path space of diffusion processes. It states that the Wiener measure (i.e.\ the distribution of Brownian motion) is quasi-invariant under a Cameron--Martin shift. In this paper, we shall derive an analogous result for subordinate Brownian motion.

    Let us recall some basic notations. Throughout this paper, we set
    $$
        [0,A]
        :=
        \begin{cases}
            \{x\,:\, 0\leq x\leq A\}, &\text{if\ \ } A<\infty,\\
            \{x\,:\, 0\leq x<\infty\}, &\text{if\ \ } A=\infty,
        \end{cases}
    $$
    and make the convention
    $$
        \int_u^v=\int_{(u,v)} \quad\text{for all $0\leq u<v\leq\infty$.}
    $$
    By $S=(S_t)_{t\in[0,T]}$, where $0<T\leq \infty$, we
    denote a non-trivial subordinator, i.e.\ an
    increasing L\'{e}vy process with $S_0=0$ and Laplace transform
    $$
        \E \e^{-uS_t}=\e^{-t\phi(u)},\quad u>0,\;t\in[0,T].
    $$
    The characteristic (Laplace) exponent $\phi$ is a Bernstein
    function with $\phi(0+)=0$; all such
    exponents are completely characterized by the following L\'evy--Khintchine formula
    \begin{equation}\label{bernstein}
        \phi(u)
        = bu+\int_0^\infty\left(1-\e^{-ux}\right)\nu(\d x), \quad u>0,
    \end{equation}
    where $b\geq0$ is the drift parameter and $\nu$ is a L\'{e}vy measure, i.e.\ a measure on $(0,\infty)$ satisfying $\int_0^\infty(x\wedge1)\,\nu(\d x)<\infty$; we use \cite{SSV} as our standard reference for Bernstein functions and subordination. If $T=\infty$ then
    $$
        S_\infty:=\lim_{t\uparrow\infty}S_t=\infty\quad \text{a.s.}
    $$
    Let
    $$
        M
        = \operatorname{\mathrm{ess\,sup}} S_T
        = \sup\left\{r>0\,:\,\Pp(S_T<r)<1\right\}.
    $$

\begin{remark}\label{eegg}
    $M$ can attain the following values:
    \begin{enumerate}
    \item[(i)]
        If $T=\infty$, then $S_\infty=\infty$ a.s.\ and $M=\infty$.

    \item[(ii)]
        If $T<\infty$ and $S_t$ is deterministic, i.e.\ $S_t=ct$ for some constant $c>0$, then $M=cT<\infty$.

    \item[(iii)]
        If $T<\infty$ and $S_t$ is non-deterministic, then $M=\infty$.

        \noindent\emph{Indeed:} Since $\nu\neq0$, there exists some finite interval $[u,v]\subset (0,\infty)$ such that $\eta:=\nu([u,v])\in(0,\infty)$. The jump times of jumps with size in the interval $[u,v]$ define a Poisson process $(N_t)_{t\in[0,T]}$ with intensity $\eta$. Since $S_T\geq uN_T$, we conclude that $\operatorname{\mathrm{ess\,sup}} S_T=\infty$.
\end{enumerate}
\end{remark}

    Let $(W_t)_{t\in[0,M]}$ be a standard $d$-dimensional Brownian motion starting from zero. The Wiener measure $\mu$, i.e.\ the distribution of $(W_t)_{t\in[0,M]}$, is a probability measure on the path space
    $$
        \W_M
        =\left\{w:[0,M]\to\R^d\,:\,\text{$w$ is continuous and $w(0)=0$}\right\},
    $$
    which is endowed with the topology of locally uniform convergence. We write
    $$
        \cmspace:=\left\{h\in\W_M\,:\,\text{$h$ is absolutely continuous and $h'\in L^2([0,M];\R^d)$}\right\}
    $$
    for the Cameron--Martin space; $\cmspace$ is a Hilbert space with the inner product
    $$
        \langle g,h\rangle _{\cmspace}
        :=\int_0^M\langle g'(t),h'(t)\rangle _{\R^d}\,\d t,\quad g,h\in\cmspace.
    $$

    Let $h\in\W_M$ and denote by $\mu_h$ the distribution of $(W_t+h(t))_{t\in[0,M]}$. Then the Cameron--Martin theorem says that $\mu$  and $\mu_h$ are equivalent (i.e.\ mutually absolutely continuous) if, and only if, $h\in\cmspace$; in this case
    $$
        \frac{\d\mu_h}{\d\mu}
        =\exp\left[\int_0^Mh'(t)\,\d W_t-\frac12\int_0^M |h'(t)|^2\,\d t\right].
    $$

    Throughout this paper, we assume that $(S_t)_{t\in[0,T]}$ is independent of the standard Brownian motion $(W_t)_{t\in[0,M]}$ on $\R^d$. The process $W_S=(W_{S_t})_{t\in[0,T]}$ is called a \emph{subordinate Brownian motion}; it is a rotationally invariant L\'{e}vy process with characteristic (Fourier) exponent (symbol) $\phi(|\xi|^2/2)
    = - \log\E \e^{i\langle \xi, W_{S_1}\rangle_{\R^d}}$.

    Inspired by the quasi-invariance property of the Wiener measure under Cameron--Martin shifts, we are interested in the following problem: {\itshape Let $\xi = (\xi_t)_{t\in[0,T]}$ be a further \textup{(}random\textup{)} function and consider the perturbation $W_S + \xi = (W_{S_t}+\xi_t)_{t\in[0,T]}$ of the subordinate Brownian motion $W_S= (W_{S_t})_{t\in [0,T]}$. For which \textup{(}random\textup{)} perturbations are the distributions of $W_S$ and $W_S+\xi$ equivalent?} Related results for compound Poisson processes can be found in \cite{WY} and \cite{Wan11}.  In this note, we assume that $\xi_t=h(S_t)$ where $h\in\W_M$ and $t\in[0,T]$. Let $\mu^S$ and $\mu_h^S$ be the distributions of $(W_{S_t})_{t\in[0,T]}$ and $(W_{S_t}+h(S_t))_{t\in[0,T]}$, respectively. We aim to find sufficient and necessary conditions so that $\mu_h^S$ is equivalent to $\mu^S$.

    Our problem can also be seen as a stability property for Bochner's subordination: Under which circumstances is the quasi-invariance on Wiener space inherited by the subordinate (i.e.\ time-changed) process?. For a deterministic  subordinator $S$, this is just the classical Cameron--Martin theorem. For a general subordinator we need to assume
    some additional conditions in order to ensure that
    quasi-invariance is preserved.

    This paper is organized in the following way: First, we establish the Cameron--Martin type theorem for subordinate Brownian motion in Section~\ref{sec2}. If $h$ is absolutely continuous, $h(0)=0$ and $\int_0^{S_T}|h'(t)|^2\,\d t<\infty$ almost surely, then $\mu^S_h$ is equivalent to $\mu^S$. The crucial point in our proof is that a functional on the path space of a subordinate Brownian motion is also a functional on the classical Wiener space. Once we have established the quasi-invariance property, we can derive in Section~\ref{sec3} an integration
    by parts formula. A natural gradient operator $D^{(\kappa)}$ is defined for $\kappa\in\R$ on the family of cylinder functions $\Fscr C_b^\infty$ through the Riesz representation theorem. Furthermore, we characterize the gradient operator $D^{(\kappa)}$; as applications we construct some natural Dirichlet forms on the path space of subordinate Brownian motion. In the final section, we present the detailed proof of a result used in Section 3, which describes the subordinator index of $S$ at the origin in terms of the underlying L\'{e}vy measure (cf. \cite{Sch98} for the corresponding result for general Feller processes).

\section{Cameron--Martin type theorem}\label{sec2}

In this section, we extend the classical Cameron--Martin theorem to subordinate Brownian motion. Let $h\in\W_M$ and write $\mu^S$ and $\mu_h^S$ for the distributions of $W_S = (W_{S_t})_{t\in[0,T]}$ and $W + h(S) = (W_{S_t}+h(S_t))_{t\in[0,T]}$, respectively. Denote by $AC([0,M];\R^d)$ the family of all absolutely continuous functions from $[0,M]$ to $\R^d$. The following Cameron--Martin type space will be important ($\kappa\in\R$):
$$
    \cmspace^{(\kappa)}
    :=\left\{h\in\W_M\cap AC([0,M];\R^d)\,:\,\int_0^M|h'(t)|^2 \left[\Pp(S_T\geq t)\right]^\kappa\,\d t<\infty\right\},
$$
which becomes a Hilbert space with the inner product
$$
    \langle g,h\rangle _{\cmspace^{(\kappa)}}
    :=\int_0^M \langle g'(t),h'(t)\rangle _{\R^d} \left[\Pp(S_T\geq t)\right]^\kappa\,\d t,
    \quad g,h\in\cmspace^{(\kappa)}.
$$
As usual, we set $0^0 := 1$. It is clear that $\cmspace^{(0)}=\cmspace$ and $\cmspace^{(\kappa)}$ is non-decreasing in $\kappa$, i.e.\ $\kappa_1\leq\kappa_2$ implies $\cmspace^{(\kappa_1)}\subset\cmspace^{(\kappa_2)}$. A direct calculation shows that
$$
    \int_0^M \langle g'(t),h'(t)\rangle _{\R^d}\left[\Pp(S_T\geq t)\right]^\kappa\,\d t
    =
    \E\left[\int_0^{S_T} \langle g'(t), h'(t)\rangle _{\R^d}\left[\Pp(S_T\geq t)\right]^{\kappa-1} \,\d t\right]
$$
for $g,h\in\cmspace^{(\kappa)}$. Therefore,
$$
    h\in\cmspace^{(1)}\quad\Longrightarrow\quad
    \int_0^{S_T}|h'(t)|^2\,\d t<\infty\quad
    \text{almost surely}.
$$

\begin{remark}
\textup{\bfseries a)}
    If either $T=\infty$ or $T<\infty$ and $M<\infty$, then $\cmspace^{(\kappa)}$ does not depend on $\kappa$ and $\cmspace^{(\kappa)}=\cmspace$ for all $\kappa\in\R$.

\medskip\noindent\textup{\bfseries b)}
    If $T<\infty$ and $M=\infty$, then $\cmspace^{(\kappa)}$ is strictly increasing in $\kappa\in\R$, see Example \ref{ggvv} below.
\end{remark}

\begin{example}\label{ggvv}
    Let $T<\infty$, $M=\infty$ and $\kappa_1<\kappa_2$. We will show that $\cmspace^{(\kappa_2)}\setminus\cmspace^{(\kappa_1)}\neq\emptyset$. Since $\Pp(S_T\geq t)\downarrow0$ as $t\uparrow\infty$, we have
    $$
        [0,\infty)=\bigcup_{m=1}^\infty B_m,
    $$
    where
    $$
        B_m
        :=\left\{t\in[0,\infty)\,:\, \frac{1}{(m+1)^3}< \left[\Pp(S_T\geq t)\right]^{\kappa_2-\kappa_1}
        \leq\frac{1}{m^3}\right\},\quad m\in\N,
    $$
    are mutually disjoint bounded Borel measurable sets. Put
    $$
        \Phi(\d t):=\left[\Pp(S_T\geq t)\right]^{\kappa_1}\,\d t.
    $$
    It is easy to see that $\Phi(B_m)<\infty$ for each $m\in\N$ and
    $$
        \#\left\{m\in\N\,:\,\Phi(B_m)>0\right\}
        = \infty.
    $$
    Pick $h\in\W_\infty\cap AC([0,\infty);\R^d)$ such that
    $$
        |h'(t)|=\sum_{m\,:\,\Phi(B_m)>0} \I_{B_m}(t)\sqrt{\frac{m}{\Phi(B_m)}},\quad t\geq0.
    $$
    Then
    \begin{align*}
        \int_0^\infty|h'(t)|^2\left[\Pp(S_T\geq t)\right]^{\kappa_2}\,\d t
        &=\sum_{m=1}^\infty \int_{B_m}|h'(t)|^2\left[\Pp(S_T\geq t)\right]^{\kappa_2-\kappa_1}\,\Phi(\d t)\\
        &\leq \sum_{m\,:\,\Phi(B_m)>0} \int_{B_m}\frac{m}{\Phi(B_m)} \cdot\frac{1}{m^3}\,\Phi(\d t)\\
        &=\sum_{m\,:\,\Phi(B_m)>0}\frac{1}{m^2}\\
        &\leq\sum_{m=1}^\infty\frac{1}{m^2}
        <\infty,
    \end{align*}
    whereas
    \begin{align*}
        \int_0^\infty|h'(t)|^2\left[\Pp(S_T\geq t)\right]^{\kappa_1}\,\d t
        &=\sum_{m\,:\,\Phi(B_m)>0} \int_{B_m}|h'(t)|^2\,\Phi(\d t)\\
        &=\sum_{m\,:\,\Phi(B_m)>0} \int_{B_m}\frac{m}{\Phi(B_m)}\,\Phi(\d t)\\
        &=\sum_{m\,:\,\Phi(B_m)>0}m=\infty.
    \end{align*}
    This means that $h\in\cmspace^{(\kappa_2)}\setminus\cmspace^{(\kappa_1)}$.
\end{example}

\begin{theorem}\label{main}
    If $h\in\W_M\cap AC([0,M];\R^d)$ and $\int_0^{S_T}|h'(t)|^2\,\d t<\infty$ almost surely \textup{(}e.g.\ if $h\in\cmspace^{(1)}$\textup{)}, then $\mu_h^S$ and $\mu^S$ are equivalent; moreover,
    $$
        \frac{\d\mu_h^S}{\d\mu^S}
        =\exp\left[\int_0^{S_T}h'(t)\,\d W_{t}-\frac12\int_0^{S_T}|h'(t)|^2\,\d t\right].
    $$
\end{theorem}

\begin{remark}
    If $S_t\equiv t$ for all $t\in[0,T]$, then $M=T\in(0,\infty]$ and Theorem \ref{main} reduces to the classical Cameron--Martin theorem.
\end{remark}

\begin{theorem}\label{ffss}
    Let $h\in\W_M$. If $\mu_h^S$ and $\mu^S$ are equivalent, then $h\in\cmspace$.
\end{theorem}

Since $\cmspace=\cmspace^{(0)}\subset\cmspace^{(1)}$, the following result is a direct consequence of Theorems \ref{main} and \ref{ffss}.

\begin{corollary}
    Assume that $h\in\W_M$. Then $\mu_h^S\sim\mu^S$ if, and only if,  $h\in\cmspace$.
\end{corollary}

For the proof of Theorem \ref{main} we need a few preparations. Note that $(W_t)_{t\in[0,M]}$ can be regarded as a process on the classical Wiener space $(\W_M,\Bscr(\W_M),\mu)$:
$$
    W_t(w):=w(t),\quad t\in[0,M],\; w\in\W_M.
$$
Let $\lambda$ be the distribution of $(S_t)_{t\in[0,T]}$, which is a probability measure on the path space
$$
    \Ss_T
    :=\left\{\ell:[0,T]\to[0,\infty)\,:\,\ell_0=0,\text{\ increasing and c\`adl\`ag}\right\},
$$
which we equip with the Skorokhod topology. Thus, the subordinator $(S_t)_{t\in[0,T]}$ can be realized as a canonical process on $(\Ss_T,\Bscr(\Ss_T),\lambda)$:
$$
    S_t(\ell):=\ell_t,\quad t\in[0,T],\;\ell\in\Ss_T.
$$
Since $S$ and $W$ are independent, $(W_{S_t})_{t\in[0,T]}$ is the canonical process on the product space $(\W_M\times\Ss_T,\Bscr(\W_M)\otimes\Bscr(\Ss_T),\mu\times\lambda)$:
$$
    W_{S_t}(w,\ell)
    :=W_{S_t(\ell)}(w)
    =w(\ell_t),
    \quad t\in[0,T],\; w\in\W_M,\; \ell\in\Ss_T.
$$
Moreover, $\mu^S$ is is a probability measure on the path space
$$
    \Omega:=\left\{w\circ\ell\,:\,w\in\W_M,\; \ell\in\Ss_T\right\}
$$
equipped with the Skorokhod topology. If $T=\infty$, we set
$$
    \ell_\infty:=\lim_{t\uparrow\infty}\ell_t,\quad \ell\in\Ss_\infty.
$$

\begin{proof}[Proof of Theorem \ref{main}]
a) We are going to show that
\begin{align*}
    &\iint\limits_{\W_M\times\Ss_T}F(w\circ\ell+h\circ\ell)\,\mu(\d w)\lambda(\d\ell)\\
    &\quad=\iint\limits_{\W_M\times\Ss_T} F(w\circ\ell)
            \exp\left[\int_0^{\ell_T}h'(t)\,\d W_t(w) -\frac12\int_0^{\ell_T}|h'(t)|^2\,\d t\right]
            \,\mu(\d w)\lambda(\d\ell)
\end{align*}
holds for every bounded measurable function $F$ on $\Omega$. Using a standard monotone class argument, it is enough to check this equality for cylinder functions of the form
$$
    F(w\circ\ell)
    =f\left(w(\ell_{t_1}),\dots,w(\ell_{t_n})\right),\quad w\circ\ell\in\Omega,
$$
where $n\in\N$, $t_1,\dots,t_n\in[0,T]$, $0<t_1<\dots<t_n$ and $f\in C_b(\R^{d\cdot n})$. Therefore, it remains to show that
\begin{equation}\label{ffgg}
\begin{aligned}
    \iint\limits_{\W_M\times\Ss_T}&f\left(w(\ell_{t_1})
    +h(\ell_{t_1}),\dots,w(\ell_{t_n})+h(\ell_{t_n})\right)\mu(\d w)\lambda(\d\ell)\\
    &=\iint\limits_{\W_M\times\Ss_T}
    f\left(w(\ell_{t_1}),\dots,w(\ell_{t_n})\right)\\
    &\qquad\qquad\quad\times\exp\left[\int_0^{\ell_T}h'(t)\,\d W_t(w)-\frac12\int_0^{\ell_T}|h'(t)|^2\,\d t\right]
    \,\mu(\d w)\lambda(\d\ell).
\end{aligned}
\end{equation}

\medskip\noindent
b)
    Fix $\ell\in\Ss_T$ such that $\ell_T\leq M$ and
    \begin{equation}\label{uuzz}
         \int_0^{\ell_T}|h'(t)|^2\,\d t<\infty.
    \end{equation}
    Then the classical Cameron--Martin theorem, applied to the bounded measurable function on $\W_{\ell_T}$:
    $$
       w\mapsto f\left(w(\ell_{t_1}),\dots, w(\ell_{t_n})\right),
    $$
    yields
    \begin{align*}
       \int_{\W_M} &f\left(w(\ell_{t_1})+h(\ell_{t_1}),\dots, w(\ell_{t_n})+h(\ell_{t_n})\right)\mu(\d w)\\
       &=\int_{\W_{\ell_T}} f\left(w(\ell_{t_1})+h(\ell_{t_1}),\dots, w(\ell_{t_n})+h(\ell_{t_n})\right)\mu(\d w)\\
       &=\int_{\W_{\ell_T}}f\left(w(\ell_{t_1}),\dots, w(\ell_{t_n})\right) \exp\left[\int_{0}^{\ell_T}h'(t)\,\d W_t(w)-\frac12\int_{0}^{\ell_T}|h'(t)|^2\,\d t\right]\mu(\d w)\\
       &=\int_{\W_M}f\left(w(\ell_{t_1}),\dots, w(\ell_{t_n})\right) \exp\left[\int_{0}^{\ell_T}h'(t)\,\d W_t(w)-\frac12\int_{0}^{\ell_T}|h'(t)|^2\,\d t\right]\mu(\d w).
    \end{align*}
    Since our assumption implies that $\ell_T\leq M$ and \eqref{uuzz} hold for $\lambda$-almost all  $\ell\in\Ss_T$, we can integrate both sides of the equality with respect to $\lambda(\d\ell)$ to obtain \eqref{ffgg}. This completes the proof.
\end{proof}

\begin{proof}[Proof of Theorem \ref{ffss}]
    Suppose that $h\notin\cmspace$. Then it is a classical result, see e.g.\ \cite[Proof of Theorem 8.1.5, p. 233--234]{Hsu}, that $\mu_h$ and $\mu$ are mutually singular. Note that the proof in \cite{Hsu} uses the time interval $[0,1]$. It is not hard to see that the method used in  \cite{Hsu} also applies to $[0,T]$ for $0<T\leq\infty$. Thus, there exists a measurable subset $A\subset\W_M$ such that
    $$
        \mu(A)=1,\quad \mu_h(A)=0.
    $$
    Let
    $$
        \tilde{A}:=\left\{w\circ\ell\in\Omega\,:\,
        w\in A,\:\ell\in\Ss_T\right\}.
    $$
    Then we have
    $$
        \mu^S(\tilde{A})=\mu(A)=1,\quad \mu_h^S(\tilde{A})=\mu_h(A)=0.
    $$
    This, however, contradicts our assumption that $\mu_h^S$ and $\mu^S$ are equivalent.
\end{proof}

\section{Integration by parts formula and gradient operator}\label{sec3}

Let $h\in\W_M$. The directional derivative of a function $F$ on $\Omega$ in direction $h$ is defined as
$$
    D_hF(w\circ\ell)
    :=\lim_{\epsilon\to 0} \frac{F(w\circ\ell+\epsilon h\circ\ell) -F(w\circ\ell)}{\epsilon},\quad w\circ\ell\in\Omega,
$$
whenever the limit exists. An important class of functions on $\Omega$ for which the above definition of $D_hF$ makes sense are the smooth cylinder functions, denoted by $\Fscr C_b^\infty$, i.e.\ the set of all functions having the form
\begin{equation}\label{jjqq}
    F(w\circ\ell)=f\left(w(\ell_{t_1}),\dots,w(\ell_{t_n})\right),
    \quad w\circ\ell\in\Omega,
\end{equation}
where $n\in\N$, $f\in C_b^\infty(\R^{d\cdot n})$ and $t_1,\dots,t_n\in[0,T]$ with $t_1<\dots<t_n$. If $F\in\Fscr C_b^\infty$ is given by \eqref{jjqq}, then it is clear that $D_hF$ exists everywhere and
\begin{equation}\label{hhgg}
    D_hF(w\circ\ell)
    =\sum_{i=1}^n \left\langle \nabla_if\left(w(\ell_{t_1}),\dots,w(\ell_{t_n})\right), h(\ell_{t_i})\right\rangle _{\R^d},
    \quad w\circ\ell\in\Omega,
\end{equation}
where $\nabla_if$ is the gradient of $f$ w.r.t.\ the $i$th variable.

First, we consider the integration by parts formula.
\begin{theorem}
    If $h\in\W_M\cap AC([0,M];\R^d)$ and $\int_0^{S_T}|h'(t)|^2\,\d t<\infty$ almost surely \textup{(} e.g.\ if $h\in\cmspace^{(1)}$\textup{)}, then for any $F,G\in\Fscr C_b^\infty$,
    $$
        \E\left[GD_hF\right]=\E\left[FD_h^*G\right],
    $$
    where
    $$
        D_h^*G:=-D_hG+G\int_0^{S_T}h'(t)\,\d W_{t}.
    $$
\end{theorem}

\begin{proof}
    Using Theorem \ref{main}, it follows that for all $\epsilon\in\R$
    \begin{align*}
        \iint\limits_{\W_M\times\Ss_T} &F(w\circ\ell+\epsilon h\circ\ell) G(w\circ\ell)\,\mu(\d w)\lambda(\d\ell)\\
        &=\iint\limits_{\W_M\times\Ss_T}F(w\circ\ell) G(w\circ\ell-\epsilon h\circ\ell)\\
        &\qquad\qquad\mbox{}\times\exp\left[\epsilon\int_0^{\ell_T}h(t)\,\d W_t(w)-\frac12\epsilon^2 \int_0^{\ell_T}|h'(t)|^2\,\d t\right] \,\mu(\d w)\lambda(\d\ell).
    \end{align*}
    Differentiating this equality w.r.t.\ $\epsilon$ and setting $\epsilon=0$, we arrive at
    \begin{align*}
        \iint\limits_{\W_M\times\Ss_T}&G(w\circ\ell)D_hF(w\circ\ell) \,\mu(\d w)\lambda(\d\ell)\\
        &=\iint\limits_{\W_M\times\Ss_T}F(w\circ\ell) \left\{-D_hG(w\circ\ell)+G(w\circ\ell) \int_0^{\ell_T}h(t)\,\d W_t(w)\right\}\,\mu(\d w)\lambda(\d\ell),
    \end{align*}
    which gives the desired assertion.
\end{proof}

Now we can investigate the gradient operator on $\Omega$.
\begin{lemma}\label{hhmm}
    Let $F\in\Fscr C_b^\infty$ and $\kappa\in\R$. Then for all $w\in\W_M$ and $\lambda$-almost all  $\ell\in\Ss_T$, the map $h\mapsto D_h F(w\circ\ell)$ is a bounded linear functional on $\cmspace^{(\kappa)}$.
\end{lemma}
\begin{proof}
    Let $F\in\Fscr C_b^\infty$ given by \eqref{jjqq}. By \eqref{hhgg}, the linearity is obvious. Since for $\lambda$-almost all  $\ell\in\Ss_T$ we have $\ell_t\leq M$ for all $t\in[0,T]$, we obtain for all $w\in\W_M$ and $\lambda$-almost all  $\ell\in\Ss_T$
    \begin{equation}\label{hh44}
    \begin{aligned}
        &\left|D_hF(w\circ\ell)\right|\leq\sum_{i=1}^n \|\nabla_if\|_\infty|h(\ell_{t_i})|\\
        &\quad\leq\sum_{i=1}^n \|\nabla_if\|_\infty\int_0^{\ell_{t_i}}
        \left[\Pp(S_T\geq t)\right]^{-\kappa/2} |h'(t)|\left[\Pp(S_T\geq t)\right]^{\kappa/2}\,\d t\\
        &\quad\leq\sum_{i=1}^n \|\nabla_if\|_\infty
        \left(
        \int_0^{\ell_{t_i}}\left[\Pp(S_T\geq t)\right]^{-\kappa}\,\d t
        \right)^{1/2}
        \left(
        \int_0^{\ell_{t_i}} |h'(t)|^2\left[\Pp(S_T\geq t)\right]^{\kappa}\,\d t
        \right)^{1/2}\\
        &\quad\leq\sum_{i=1}^n \|\nabla_if\|_\infty
        \left(
        \int_0^{\ell_{t_i}}\left[\Pp(S_T\geq t)\right]^{-\kappa}\,\d t
        \right)^{1/2}
        \|h\|_{\cmspace^{(\kappa)}}.
    \end{aligned}
    \end{equation}
    To complete the proof, it remains to note that
    $$
        \int_0^{\ell_{t_i}}\left[\Pp(S_T\geq t)\right]^{-\kappa}\,\d t
        \leq\left[1\vee\left[\Pp(S_T\geq\ell_{t_i})\right]^{-\kappa}\right]\ell_{t_i}
        \leq\left[1\vee\left[\Pp(S_T\geq\ell_{T})\right]^{-\kappa}\right]\ell_{t_i}
    $$
    and
    $$
        \Pp(S_T\geq\ell_{T})>0
    $$
    for $\lambda$-almost all  $\ell\in\Ss_T$.
\end{proof}

Let $F\in\Fscr C_b^\infty$ and $\kappa\in\R$. Combining Lemma \ref{hhmm} with the Riesz representation theorem we find for all
$w\in\W_M$ and $\lambda$-almost all  $\ell\in\Ss_T$ that there exists a unique $D^{(\kappa)}F(w\circ\ell)\in\cmspace^{(\kappa)}$ such that
$$
    \left\langle D^{(\kappa)}F(w\circ\ell),h\right\rangle _{\cmspace^{(\kappa)}}
    =D_hF(w\circ\ell), \quad F\in\Fscr C_b^\infty,\; h\in\cmspace^{(\kappa)}.
$$
For simplicity, we write $DF$ instead of $D^{(0)}F$. If $F\in\Fscr C_b^\infty$ is given by \eqref{jjqq}, then it is easy to see that for all $w\in\W_M$ and $\lambda$-almost all  $\ell\in\Ss_T$ we have
\begin{gather*}
    D^{(\kappa)}F(w\circ\ell)(t)
    =\sum_{i=1}^n\nabla_i
    f\left(w(\ell_{t_1}),\dots,w(\ell_{t_n})\right)
    \int_0^{t\wedge\ell_{t_i}} \left[\Pp(S_T\geq s)\right]^{-\kappa}\,\d s,\quad t\in[0,M],\\
    \left\|D^{(\kappa)}F(w\circ\ell)\right\|_{\cmspace^{(\kappa)}}^2
    =\sum_{i,j=1}^n\left\langle\nabla_i
    f\left(w(\ell_{t_1}),\dots,w(\ell_{t_n})\right),
    \nabla_j f\left(w(\ell_{t_1}),\dots,w(\ell_{t_n})\right)
    \right\rangle_{\R^d}\\
    \times
     \int_0^{\ell_{t_i}\wedge\ell_{t_j}}
     \left[\Pp(S_T\geq s)\right]^{-\kappa}\,\d s.
\end{gather*}
In particular, for $F\in\Fscr C_b^\infty$ having the form \eqref{jjqq}, it holds that for all $w\in\W_M$ and $\lambda$-almost all  $\ell\in\Ss_T$
\begin{gather*}
    DF(w\circ\ell)(t)
    =\sum_{i=1}^n\left(t\wedge\ell_{t_i}\right)\nabla_i
    f\left(w(\ell_{t_1}),\dots,w(\ell_{t_n})\right)
    ,\quad t\in[0,M],\\
     \left\|DF(w\circ\ell)\right\|_{\cmspace}^2
     =\sum_{i=1}^n\left(\ell_{t_{i}}-\ell_{t_{i-1}}
     \right)\left|\sum_{j=i}^n
     \nabla_j
     f\left(w(\ell_{t_1}),\dots,w(\ell_{t_n})\right)
     \right|^2,
\end{gather*}
where $t_0:=0$.

Recall that $\nu$ is the L\'{e}vy measure of the subordinator $(S_t)_{t\in[0,T]}$. We will use the following integrability condition:
\begin{align}\label{H}
    \int_1^\infty x^{p/2}\,\nu(\d x)<\infty,
\tag{\textup{H}$_p$}
\end{align}
where $p>0$. In fact, since $\nu(1,\infty)<\infty$, \eqref{H} is automatically satisfied for $p\leq0$.

\begin{remark}
    It is well known that \eqref{H} is equivalent to $\E S_t^{p/2}<\infty$ for some (or all) $t\in[0,T]$, cf.\ \cite[Theorem 25.3]{Sato}.
\end{remark}

Let us introduce the following index of $S$ at the origin:
$$
     \sigma_0
     :=\sup\left\{\alpha\geq0\,:\,\lim_{u\downarrow0}\frac{\phi(u)}{u^\alpha}=0\right\}.
$$
Noting that
$$
    \lim_{u\downarrow0}\frac{\phi(u)}{u}
    =\lim_{u\downarrow0}\phi'(u)
    =b+\lim_{u\downarrow0}\int_0^\infty x\e^{-ux}\,\nu(\d x)
    =b+\int_0^\infty x\,\nu(\d x)
    \in(0,\infty],
$$
it follows that
$$
    0\leq\sigma_0\leq1.
$$

The following useful proposition is the subordinator counterpart of a result on general Feller processes from \cite{Sch98}. We defer its proof to the appendix (Section~\ref{sec4}).

\begin{proposition}\label{index}
    Let $\phi$ be given by \eqref{bernstein}. Then
    \begin{equation}\label{n1n}
        \sigma_0
        =\sup\left\{\alpha\geq0\,:\,\limsup_{u\downarrow0}\frac{\phi(u)}{u^\alpha}<\infty\right\}
        =\sup\left\{0\leq\rho\leq1\,:\,\int_1^\infty x^\rho\,\nu(\d x)<\infty\right\}.
    \end{equation}
\end{proposition}

\begin{remark}\label{control}
    \textbf{\upshape a)}
        Clearly, $p/2<\sigma_0$ implies \eqref{H}; conversely \eqref{H} entails that
        either $p/2\leq\sigma_0$ or $p/2>1$.
        In particular, \textup{(H$_2$)} implies $\sigma_0=1$.

    \medskip\noindent\textbf{\upshape b)}
        Since
        $$
            \phi'(0+):=\lim_{u\downarrow0}\phi'(u)
            =b+\int_{(0,1]}x\,\nu(\d x) +\int_1^\infty x\,\nu(\d x),
        $$
        \textup{(H$_2$)} is equivalent to $\phi'(0+)<\infty$.

    \medskip\noindent\textbf{\upshape c)}
        \textup{(H$_2$)} is strictly stronger than $\sigma_0=1$. An example of a Bernstein function satisfying $\sigma_0=1$ but not \textup{(H$_2$)} is (cf. \cite[p. 316]{SSV})
        $$
            \phi(u)=u\log\left(1+\frac1u\right),\quad u>0.
        $$
        This function is even a complete Bernstein function. To see our claim, note that
        $$
            \sigma_0=1
            \iff
            \lim_{u\downarrow0}\frac{\phi(u)}{u^\alpha}=0 \quad
            \text{for all $\alpha\in(0,1)$}
        $$
        and
        $$
           \text{\textup{(H$_2$)} fails}
           \iff
           \phi'(0+)=\infty
           \iff
           \lim_{u\downarrow0}\frac{\phi(u)}{u}=\infty \quad
           \text{(L'H\^{o}spital's rule)}.
        $$
\end{remark}

\begin{lemma}\label{int23}
    Let $T<\infty$ and $M=\infty$. If $\sigma_0>0$ and
    $\theta>1/\sigma_0$, then
    $$
        \int_1^\infty\left[\Pp(S_T\geq t)\right]^{\theta}
        \,\d t<\infty.
    $$
\end{lemma}

\begin{remark}
    Let $T<\infty$ and $S$ be an $\alpha$-stable subordinator ($0<\alpha<1$). Obviously, $\sigma_0=\alpha$. It is well known that
    $$
        \Pp(S_T\geq t)\asymp t^{-\alpha},\quad t\geq1.
    $$
    Therefore,
    $$
        \int_1^\infty\left[\Pp(S_T\geq t)\right]^{\theta}\,\d t<\infty
        \iff
        \int_1^\infty t^{-\alpha\theta}\,\d t<\infty
        \iff
        \theta>\frac1\alpha=\frac{1}{\sigma_0}.
    $$
    This means that Lemma \ref{int23} is sharp for  $\alpha$-stable subordinators.
\end{remark}

\begin{proof}[Proof of Lemma \ref{int23}]
    Let $t\geq 1$; using
    $$
        \I_{[t,\infty)}(S_T)\leq\frac{2S_T}{S_T+t}
    \quad\text{and}\quad
        \frac{x}{x+t}=t\int_0^\infty\left(1-\e^{-ux}\right)\e^{-tu}\,\d u,\quad x\geq0,
    $$
    together with Tonelli's theorem, we find
    \begin{align*}
        \Pp(S_T\geq t)
        &= \E\left[\I_{[t,\infty)}(S_T)\right]\\
        &\leq 2\E\left[\frac{S_T}{S_T+t}\right]\\
        &= 2t\E\left[\int_0^\infty\left(1-\e^{-uS_T}\right)\e^{-tu}\,\d u\right]\\
        &= 2t\int_0^\infty\left(1-\e^{-T\phi(u)}\right)\e^{-tu}\,\d u\\
        &= 2\int_0^\infty\left(1-\e^{-T\phi(u/t)}\right)\e^{-u}\,\d u\\
        &\leq 2T\int_0^\infty\phi\left(\frac ut\right)\e^{-u}\,\d u,
    \end{align*}
    where the last estimate follows from the elementary inequality
    $$
        1-\e^{-x}\leq x,\quad x\in\R.
    $$
    Pick $\alpha\in\left(1/\sigma_0,\theta\right)$. By the first equality in \eqref{n1n}, there exists a constant $c=c(\alpha)>0$ such that
    $$
        \phi(u)\leq cu^{1/\alpha},\quad 0<u\leq1.
    $$
    On the other hand, we have for all $u\geq1$ that
    \begin{align*}
        \phi(u)
        &\leq bu+\left(\int_0^1x\,\nu(\d x)\right)u+\nu(x\geq1)\\
        &\leq \left[b+\int_0^1x\,\nu(\d x)+\nu(x\geq1)\right]u
        \,=:\,C_1u.
    \end{align*}
    Thus, we get for $t\geq1$---note that $1/\alpha\in(0,1)$---:
    \begin{align*}
        \Pp(S_T\geq t)
        &\leq 2T\int_0^tc\left(\frac ut\right)^{1/\alpha}\e^{-u}\,\d u + 2T\int_t^\infty C_1\frac ut\e^{-u}\,\d u\\
        &\leq 2T\left(c\int_0^tu^{1/\alpha}\e^{-u}\,\d u + C_1\int_t^\infty u\e^{-u}\,\d u\right) t^{-1/\alpha}\\
        &\leq 2T\left(c\Gamma\left(\frac1\alpha+1\right) + C_1\right)t^{-1/\alpha}\\
        &=:C_2t^{-1/\alpha}.
    \end{align*}
    This, together with $\theta/\alpha>1$, implies that
    \begin{gather*}
        \int_1^\infty\left[\Pp(S_T\geq t)\right]^{\theta}\,\d t
        \leq C_2^\theta\int_1^\infty t^{-\theta/\alpha}\,\d t
        <\infty. \qedhere
    \end{gather*}
\end{proof}

From the point view of functional analysis, the gradient operator is only useful if it is closable in some Banach space. To show this, the following two conditions will be used:
\begin{align}
\tag{A1}\label{eqA1}
    &\text{\eqref{H} holds for some $p>0$;}\\
\tag{A2}\label{eqA2}
    &T<\infty,\; M=\infty,\; p\in(0,2],\;\sigma_0>0 \text{\ \ and\ \ } \kappa<1-1/\sigma_0.
\end{align}

\begin{remark}
\textup{\bfseries a)}
    Since $\sigma_0\leq1$, we know that $\kappa<0$ is necessary for \eqref{eqA2}.

\medskip\noindent\textup{\bfseries b)}
    Assume that $T<\infty$ and $M=\infty$. Then \eqref{eqA1} with $p\in(0,2]$ is strictly stronger than \eqref{eqA2}. \emph{Indeed:} First note that by Proposition \ref{index} \eqref{eqA1} with $p\in(0,2]$ implies $\sigma_0=p/2>0$, and so \eqref{eqA2} is fulfilled with $\kappa<1-2/p$; moreover, for an $\alpha$-stable subordinator ($\alpha\in(0,1)$) \eqref{eqA2} holds with $\sigma_0=\alpha>0$, $\kappa<1-1/\alpha$ and all $p\in(0,2]$, while \eqref{eqA1} holds if and only if $p<2\alpha<2$.
\end{remark}

\begin{lemma}\label{hh09}
    Let $F\in\Fscr C_b^\infty$.

    \medskip\noindent\textup{\bfseries a)}
        If \eqref{eqA1} holds, then $D_hF\in L^p(\mu^S)$ for any $h\in\cmspace$ and $DF\in L^p(\Omega\to\cmspace;\mu^S)$.

   \medskip\noindent\textup{\bfseries b)}
        If \eqref{eqA2} holds, then $D_hF\in L^p(\mu^S)$ for any $h\in\cmspace^{(\kappa)}$ and $D^{(\kappa)}F\in L^p(\Omega\to\cmspace^{(\kappa)};\mu^S)$.
\end{lemma}

\begin{proof}
    By \eqref{hh44} and the elementary inequality
    $$
        \left(\sum_{i=1}^na_i\right)^p\leq n^{(p-1)\vee 0} \sum_{i=1}^na_i^p,\quad a_i\geq0,
    $$
    we obtain for any $\kappa\in\R$ and $w\in\W_M$ and $\lambda$-almost all  $\ell\in\Ss_T$
    $$
        \left|D_hF(w\circ\ell)\right|^p
        \leq n^{(p-1)^+}\|h\|_{\cmspace^{(\kappa)}}^p\sum_{i=1}^n \|\nabla_if\|_\infty^p
        \left(\int_0^{\ell_{t_i}}\left[\Pp(S_T\geq t)\right]^{-\kappa}\,\d t
        \right)^{p/2},\quad h\in\cmspace^{(\kappa)},
    $$
    \begin{align*}
        \left\|D^{(\kappa)}F(w\circ\ell)\right\|_{\cmspace^{(\kappa)}}^p
        &=\sup_{\|h\|_{\cmspace^{(\kappa)}}\leq1}\left|D_hF(w\circ\ell)\right|^p\\
        &\leq n^{(p-1)\vee 0}\sum_{i=1}^n \|\nabla_if\|_\infty^p
        \left(
        \int_0^{\ell_{t_i}}\left[\Pp(S_T\geq t)\right]^{-\kappa}\,\d t
        \right)^{p/2}.
    \end{align*}

   \medskip\noindent\textbf{a)}
   Assume that \eqref{eqA1} holds. Let $\kappa=0$. Then we have
    $$
        \int_{\Ss_T}\left(
        \int_0^{\ell_{t_i}}\left[\Pp(S_T\geq t)\right]^{-\kappa}\,\d t
        \right)^{p/2} \lambda(\d\ell)
        =\int_{\Ss_T}\ell_{t_i}^{p/2}\,\lambda(\d\ell)
        =\E S_{t_i}^{p/2}<\infty,
    $$
    so that the first assertion follows.

    \medskip\noindent\textbf{b)}
    Assume that \eqref{eqA2} holds. We use the Jensen inequality and Lemma \ref{int23} to get
    \begin{align*}
        \int_{\Ss_T}\left(
            \int_0^{\ell_{t_i}}\left[\Pp(S_T\geq t)\right]^{-\kappa}\,\d t
        \right)^{p/2}\,\lambda(\d\ell)
        &\leq\int_{\Ss_T}\left(
            \int_0^{\ell_{T}}\left[\Pp(S_T\geq t)\right]^{-\kappa}\,\d t
        \right)^{p/2}\,\lambda(\d\ell)\\
        &\leq
        \left(\int_{\Ss_T}\left(
            \int_0^{\ell_{T}}\left[\Pp(S_T\geq t)\right]^{-\kappa}\,\d t
        \right)\lambda(\d\ell)\right)^{p/2}\\
        &=\left(
            \int_0^\infty\left[\Pp(S_T\geq t)\right]^{1-\kappa}\,\d t
        \right)^{p/2}\\
        &\leq\left(
            1+\int_1^\infty\left[\Pp(S_T\geq t)\right]^{1-\kappa}\,\d t
        \right)^{p/2}\\
        &<\infty.
        \qedhere
    \end{align*}
\end{proof}

For $p\geq 1$ we set
$$
    L^{p+}(\mu^S) :=\bigcup_{p<r\leq\infty}L^{r}(\mu^S).
$$
As usual, we make the convention $1/0:=\infty$.

\begin{theorem}\label{00vc}
    Assume that \eqref{eqA1} holds with some $p\geq1$
        \textup{(}resp.\ \eqref{eqA2} holds with some $p\in[1,2]$\textup{)}.
    Let $p/(p-1)\leq q\leq\infty$. For any $h\in\cmspace$
        \textup{(}resp.\ $h\in \cmspace^{(\kappa)}$\textup{)},
    the directional derivative operator $D_h:\Fscr C_b^\infty\to L^p(\mu^S)$ is closable in $L^q(\mu^S)$. Denote by $D_h$ its closure and let $D_h^*$ be its adjoint. Then
    $$
        \Dscr(D_h)\cap L^{p+}(\mu^S) \subset\Dscr(D_h^*)
    $$
    and for $G\in\Dscr(D_h)\cap L^{p+}(\mu^S)$,
    \begin{equation}\label{cc12}
        D^*_hG=-D_hG+G\int_0^{S_T}h'(t)\,\d W_t.
    \end{equation}
\end{theorem}

\begin{proof}
a)
    Assume that \eqref{eqA1} holds for some $p\geq1$.

    \medskip\noindent
a1)
    Let $\{F_n\}_{n\in\N}\subset\Fscr C_b^\infty$ be a sequence such that $F_n\to 0$ in $L^q(\mu^S)$ and $D_hF_n\to Z$ in $L^p(\mu^S)$ as $n\to\infty$. In order to prove the closability of $D_h$, we have to show that $Z=0$.

    Fix $G\in\Fscr C_b^\infty$. According to Lemma \ref{hh09}, $D_hG\in L^p(\mu^S)$. On the other hand, by Burkholder's inequality,
    we have  for all $\theta\in(0,\infty)$
    \begin{equation}\label{bur}
    \begin{aligned}
    \E\left[\left|\int_0^{S_T}h'(t)\,\d W_t\right|^\theta\right]
        &= \int_{\Ss_T}\E \left[\left|\int_0^{\ell_T}h'(t)\,\d W_t\right|^\theta\right]\,\lambda(\d\ell)\\
        &\leq\int_{\Ss_T}C_\theta\left(\int_0^{\ell_T}|h'(t)|^2\,\d t \right)^{\theta/2}\,\lambda(\d\ell)\\
        &\leq C_\theta\|h\|_{\cmspace}^\theta,
    \end{aligned}
    \end{equation}
    where $C_\theta$ is a positive constant depending on $\theta$. Thus,
    $$
        D_h^*G=-D_hG+G\int_0^{S_T}h'(t)\,\d W_t\in L^p(\mu^S).
    $$
    The conjugate H\"{o}lder exponent $q'$ of $q$ satisfies $1\leq q'\leq p$, and this implies that $D_h^*G\in L^{q'}(\mu^S)$. Therefore, we obtain
    $$
        \E\left[GZ\right]
        =\lim_{n\to\infty}\E\left[GD_hF_n\right]
        =\lim_{n\to\infty}\E\left[F_nD_h^*G\right]
        =0.
    $$
    Since $G\in\Fscr C_b^\infty$ is arbitrary and $\Fscr C_b^\infty$ is dense in $L^q(\mu^S)$, we get $Z=0$.

    \medskip\noindent
a2)
    We proceed as in \cite[Proof of Theorem 5.2]{Hs95}. Let $G\in\Dscr(D_h)\cap L^{r}(\mu^S)$ for some $r\in(p,\infty]$. Since $D_{h}:\Fscr C_b^\infty\to L^p(\mu^S)$ is closable in $L^q(\mu^S)$, there exists a sequence $\{G_n\}_{n\in\N}\subset\Fscr C_b^\infty$ such that $G_n\to G$ in $L^q(\mu^S)$ and $D_hG_n\to D_hG$ in $L^p(\mu^S)$ as $n\to\infty$. With \eqref{bur} it is easy to see that
    $$
        D_h^*G_n
        =-D_hG_n+G_n\int_0^{S_T}h'(t)\,\d W_t
        \xrightarrow[n\to\infty]{} -D_hG+G\int_0^{S_T}h'(t)\,\d W_t$$
    in $L^1(\mu^S)$. Therefore, for any $F\in\Fscr C_b^\infty$, by $D_hF\in L^p(\mu^S)\subset L^{q'}(\mu^S)$, one has
    \begin{align*}
        \E\left[GD_hF\right]
        &=\lim_{n\to\infty} \E\left[G_nD_hF\right]\\
        &=\lim_{n\to\infty} \E\left[FD_h^*G_n\right]\\
        &=\E\left[F\left\{-D_hG+G\int_0^{S_T}h'(t)\,\d W_t\right\}\right].
    \end{align*}
    Since $G\in L^r(\mu^S)$ and \eqref{bur}, together with H\"{o}lder's inequality, imply that
    $$
         G\int_0^{S_T}h'(t)\,\d W_t\in L^p(\mu^S),
    $$
    we get $G\in\Dscr(D_h^*)$; in particular, \eqref{cc12} follows.

\medskip\noindent
b) If \eqref{eqA2} holds for some $p\in[1,2]$, one can prove the claim as in the first case; we only need to replace \eqref{bur} by
    \begin{align*}
        \E\left[\left|\int_0^{S_T}h'(t)\,\d W_t\right|^\theta\right]
        &\leq\int_{\Ss_T}C_\theta\left(\int_0^{\ell_T}|h'(t)|^2\,\d t \right)^{\theta/2}\,\lambda(\d\ell)\\
        &\leq C_\theta\left(\int_0^\infty|h'(t)|^2\left[\Pp(S_T\geq t)\right]^\kappa\,\d t\right)^{\theta/2}\\
        &=C_\theta\|h\|_{\cmspace^{(\kappa)}}^\theta,
    \end{align*}
    where we used, in the second inequality, $\kappa<0$.
\end{proof}

\begin{theorem}\label{3cc}
    Assume that \eqref{eqA1} holds for some $p\geq1$
        \textup{(}resp.\ \eqref{eqA2} holds for some $p\in[1,2]$\textup{)}.
    Let $p/(p-1)\leq q\leq\infty$. The gradient operator $D:\Fscr C_b^\infty\to L^p(\Omega\to\cmspace;\mu^S)$
        \textup{(}resp.\ $D^{(\kappa)}:\Fscr C_b^\infty\to L^p(\Omega\to\cmspace^{(\kappa)};\mu^S)$\textup{)}
    is closable in $L^q(\mu^S)$. Denote its closure again by $D$
        \textup{(}resp.\ $D^{(\kappa)}$\textup{)}.
    For any $h\in\cmspace$, $\Dscr(D)\subset\Dscr(D_h)$; if $F\in\Dscr(D)$ then $D_hF=\langle DF,h\rangle _{\cmspace}$ in $L^p(\mu^S)$
        \textup{(}resp.\ for any $h\in\cmspace^{(\kappa)}$, $\Dscr(D^{(\kappa)})\subset\Dscr(D_h)$\textup{)};
    if $F\in\Dscr(D^{(\kappa)})$ then $D_hF=\langle D^{(\kappa)}F,h\rangle_{\cmspace^{(\kappa)}}$ in $L^p(\mu^S)$\textup{)}.
\end{theorem}

\begin{proof}
    We only prove the statement for the case that \eqref{eqA1}
    holds for some $p\geq1$, since the proof for the other
    case is essentially similar.

    \medskip\noindent
a)
    Let $\{F_n\}_{n\in\N}\subset\Fscr C_b^\infty$ be a sequence such that $F_n\to0$ in $L^q(\mu^S)$ and $DF_n\to Y$ in $L^p(\Omega\to\cmspace;\mu^S)$ as $n\to\infty$. We will prove that $Y=0$. Fix an orthonormal basis $\{h_i\}_{i\in\N}$ of $\cmspace$ and define
    $$
        \Cscr(\cmspace)
        :=\left\{\sum_{i=1}^mG_ih_i\,:\,m\in\N, \; G_i\in\Fscr C_b^\infty\right\}.
    $$
    It is easy to see that $\Cscr\subset L^q(\Omega\to\cmspace;\mu^S)$ is a dense subset. Pick $G=\sum_{i=1}^mG_ih_i\in\Cscr(\cmspace)$. By the orthogonal expansion of $DF_n$ in the orthonormal basis $\{h_i\}_{i\in\N}$ and the closability of $D_{h_i}:\Fscr C_b^\infty\to L^p(\mu^S)$ in $L^q(\mu^S)$, cf.\ Theorem \ref{00vc}, we get
    \begin{align*}
            \E\left\langle G,Y\right\rangle _{\cmspace}
            &= \lim_{n\to\infty}\E\left\langle G,DF_n\right\rangle _{\cmspace}\\
            &=\lim_{n\to\infty}\E\left[\sum_{i=1}^m\sum_{j=1}^\infty \langle h_i,h_j\rangle _{\cmspace}G_i D_{h_j}F_n\right]\\
            &=\lim_{n\to\infty}\sum_{i=1}^m\E\left[G_iD_{h_i}F_n\right]\\
            &=\sum_{i=1}^m\lim_{n\to\infty}\E\left[G_iD_{h_i}F_n\right] = 0.
    \end{align*}
    Since $G\in\Cscr$ is arbitrary and $\Cscr$ is dense in $L^q(\Omega\to\cmspace;\mu^S)$, this implies that $Y=0$. Thus, $D$ is closable.

    \medskip\noindent
b)
    We will adopt the idea used in \cite[Proof of Proposition 5.3]{Hs95}. Let $h\in\cmspace$ and $F\in\Dscr(D)$. Then there exists a sequence $\{F_n\}_{n\in\N}\subset\Fscr C_b^\infty$ such that $F_n\to F$ in $L^q(\mu^S)$ and $DF_n\to DF$ in $L^p(\Omega\to\cmspace;\mu^S)$ as $n\to\infty$. Since
    $$
        |D_hF_n-D_hF_m|^p=\left|\langle D(F_n-F_m),h\rangle _{\cmspace}\right|^p
        \leq\|D(F_n-F_m)\|_{\cmspace}^p\|h\|_{\cmspace}^p
    $$
    holds for all $m,n\in\N$, we conclude that $\{D_hF_n\}_{n\in\N}$ is a Cauchy sequence in $L^p(\mu^S)$. By the closedness of $D_h$, we have $F\in\Dscr(D_h)$ and $D_hF_n\to D_hF$ in $L^p(\mu^S)$ as $n\to\infty$. To finish the proof, it remains to note that
    \begin{align*}
        \|D_hF-\langle DF,h\rangle _{\cmspace}\|_{L^p(\mu^S)}
        &\leq \|D_hF-D_hF_n\|_{L^p(\mu^S)} + \|\langle DF_n,h\rangle -\langle DF,h\rangle _{\cmspace}\|_{L^p(\mu^S)}\\
        &\leq \|D_hF-D_hF_n\|_{L^p(\mu^S)} + \|DF_n-DF\|_{L^p(\Omega\to\cmspace;\mu^S)}\|h\|_{\cmspace}\\
        &\xrightarrow[n\to\infty]{} 0
        \qedhere.
    \end{align*}
\end{proof}

If \eqref{eqA1} holds for $p=2$, then by Lemma \ref{hh09}\,a), we can define a symmetric quadratic form on $\Fscr C_b^\infty$ in the following way:
$$
    \Escr(F,G)
    :=\int_\Omega\langle DF,DG\rangle _{\cmspace}\,\d\mu^S<\infty,
    \quad F,G\in\Fscr C_b^\infty.
$$
If $F\in\Fscr C_b^\infty$ is of the form \eqref{jjqq}, then
$$
    \Escr(F,F)
    =\E\left[\sum_{i=1}^n\left(S_{t_i}- S_{t_{i-1}}\right)\left|\sum_{j=i}^n \nabla_jf\left(W_{S_{t_1}},\dots, W_{S_{t_n}}\right)\right|^2\right],
$$
where $t_0:=0$.

Similarly, if \eqref{eqA2} holds for $p=2$, then with Lemma \ref{hh09}\,b) we can define a symmetric quadratic form on $\Fscr C_b^\infty$:
$$
    \Escr^{(\kappa)}(F,G)
    :=\int_\Omega\left\langle D^{(\kappa)}F, D^{(\kappa)}G\right\rangle _{\cmspace^{(\kappa)}}\,\d\mu^S<\infty,
    \quad F,G\in\Fscr C_b^\infty.
$$
Moreover, for $F\in\Fscr C_b^\infty$ given by \eqref{jjqq}, one has
\begin{align*}
    \Escr^{(\kappa)}(F,F)
    =\E\Bigg[\sum_{i,j=1}^n
    &\left\langle\nabla_i f\left(W_{S_{t_1}},\dots,W_{S_{t_n}}\right),\nabla_j f\left(W_{S_{t_1}},\dots,W_{S_{t_n}}\right)\right\rangle_{\R^d}\\
    &\quad\times\int_0^{S_{t_i\wedge t_j}}\left[\Pp(S_T\geq s)\right]^{-\kappa}\,\d s\Bigg].
\end{align*}

As an immediate consequence of Theorem \ref{3cc} with $p=q=2$, we obtain the following result concerning the Dirichlet form on $L^2(\mu^S)$ (see \cite{Fuk, RM} for more details on the theory of Dirichlet forms).

\begin{proposition}\label{j85cv}
    If \eqref{eqA1}
        \textup{(}resp.\ \eqref{eqA2}\textup{)}
    holds for $p=2$, then the form $(\Escr,\Fscr C_b^\infty)$
        \textup{(}resp.\ $(\Escr^{(\kappa)},\Fscr C_b^\infty)$\textup{)}
    is closable in $L^2(\mu^S)$, and the closure $(\Escr,\Dscr(\Escr))$
        \textup{(}resp.\ $(\Escr^{(\kappa)},\Dscr(\Escr^{(\kappa)}))$\textup{)}
    is a conservative symmetric Dirichlet form on $L^2(\mu^S)$.
\end{proposition}

We close this section by pointing out that it might be interesting (and also challenging) to consider various functional inequalities (cf. \cite{Wbook}) for the Dirichlet forms derived in Proposition \ref{j85cv}.

\section{Appendix}\label{sec4}

In this section, we establish Proposition \ref{index}.
\begin{proof}[Proof of Proposition \ref{index}]
a)
   Denote by $\sigma'_0$ and $\sigma''_0$ the second and
   the third term in \eqref{n1n}, respectively.
   Clearly, $\sigma'_0\geq\sigma_0$. For any $\alpha<\sigma'_0$, pick $\alpha'\in(\alpha,\sigma'_0)$. According to the definition of $\sigma'_0$, one has
   $$
       \limsup_{u\downarrow 0} \frac{\phi(u)}{u^{\alpha'}}<\infty,
   $$
   and so
   $$
       \frac{\phi(u)}{u^{\alpha}}
       =\frac{\phi(u)}{u^{\alpha'}}u^{\alpha'-\alpha}
       \rightarrow0\quad \text{as $u\downarrow0$}.
   $$
   This means that $\alpha\leq\sigma_0$. Letting $\alpha\uparrow\sigma'_0$, we get $\sigma'_0\leq\sigma_0$. Therefore, $\sigma'_0=\sigma_0$.

\medskip\noindent
b)
    We prove that $\sigma_0\leq\sigma''_0$. Without loss of generality, we may assume that $0<\sigma_0\leq1$. Fix any $\alpha\in(0,\sigma_0)\subset(0,1)$ and pick $\alpha'\in(\alpha,\sigma_0)$. By the definition of $\sigma_0$,
    $$
      \lim_{u\downarrow0}\frac{\phi(u)}{u^{\alpha'}}=0
    $$
    and so there exists a constant $c=c(\alpha')>0$ such that
    \begin{equation}\label{smallt}
        \phi(u)\leq cu^{\alpha'},\quad 0<u<1.
    \end{equation}
    Recall the identity
    $$
        x^\alpha
        =\frac{\alpha}{\Gamma(1-\alpha)}\int_0^\infty\left(1-\e^{-xu}\right)u^{-\alpha-1}\,\d u,
        \quad x>0.
    $$
    Combining this with Tonelli's theorem, the following inequality
    $$
        \int_1^\infty\left(1-\e^{-xu}\right)\nu(\d x)
        \leq\nu(1,\infty)\wedge\left(
        \int_0^\infty\left(1-\e^{-xu}\right)\nu(\d x)\right)
        \leq\nu(1,\infty)\wedge
        \phi(u),\quad u>0
    $$
    and \eqref{smallt}, we obtain that
    \begin{align*}
        \int_1^\infty x^\alpha\,\nu(\d x)
        &=\frac{\alpha}{\Gamma(1-\alpha)} \int_0^\infty\left(\int_1^\infty\left(1-\e^{-xu}\right) \nu(\d x)\right)u^{-\alpha-1}\,\d u\\
        &\leq\frac{\alpha}{\Gamma(1-\alpha)} \int_0^\infty\big[\nu(1,\infty)\wedge\phi(u)\big] u^{-\alpha-1}\,\d u\\
        &\leq\frac{\alpha}{\Gamma(1-\alpha)} \int_0^1\phi(u)u^{-\alpha-1}\,\d u
         + \frac{\alpha\nu(1,\infty)}{\Gamma(1-\alpha)} \int_1^\infty u^{-\alpha-1}\,\d u\\
        &\leq\frac{\alpha c}{\Gamma(1-\alpha)} \int_0^1u^{\alpha'-\alpha-1}\,\d u
         +\frac{\alpha\nu(1,\infty)}{\Gamma(1-\alpha)} \int_1^\infty u^{-\alpha-1}\,\d u
        \,<\,\infty.
   \end{align*}
   This implies that $\alpha\leq\sigma''_0$. Since $\alpha<\sigma_0$ is arbitrary, we conclude that $\sigma_0\leq\sigma''_0$.

\medskip\noindent
c)
    It remains to show that $\sigma''_0\leq\sigma_0$. As in part b), we can assume that $\sigma''_0\in(0,1]$. Fix any $\rho\in(0,\sigma''_0)$ and pick $\rho'\in(\rho,\sigma''_0)\subset(0,1)$. Then
    $$
        \int_1^\infty x^{\rho'}\,\nu(\d x)<\infty.
    $$
    Assume that $u\in(0,1)$. Since
    \begin{gather*}
       1-\e^{-xu}
       <
          \begin{cases}
          xu, &\text{if\ \ } 0<x\leq1,\\
          xu\leq(xu)^{\rho'}, &\text{if\ \ } 1<x\leq u^{-1},\\
          1<(xu)^{\rho'}, &\text{if\ \ } x>u^{-1}
          \end{cases}
       \implies
       1-\e^{-xu}
       <
          \begin{cases}
          xu, &\text{if\ \ } 0<x\leq1,\\
          (xu)^{\rho'}, &\text{if\ \ } x>1,
          \end{cases}
    \end{gather*}
    it follows that
    \begin{align*}
        \phi(u)
        &\leq bu+\int_{(0,1]}xu\,\nu(\d x)
           +\int_1^\infty(xu)^{\rho'}\,\nu(\d x)\\
        &\leq\left[b+\int_{(0,1]}x\,\nu(\d x)
           +\int_1^\infty x^{\rho'}\,\nu(\d x)\right]
             u^{\rho'}
        \,=:\,C(\rho')u^{\rho'}.
   \end{align*}
   From this we get
   $$
       \limsup_{u\downarrow0}\frac{\phi(u)}{u^\rho}
       \leq\limsup_{u\downarrow0}C(\rho')u^{\rho'-\rho}=0.
   $$
   Consequently, we obtain $\rho\leq\sigma_0$ and then $\sigma''_0\leq\sigma_0$ by letting $\rho\uparrow\sigma''_0$.
\end{proof}

\end{document}